\documentclass[11pt]{amsart}
\usepackage{amsmath,amssymb,amsthm,times}
\usepackage{microtype}
\usepackage{graphicx}
\usepackage{subcaption}
\usepackage[colorlinks=true,linkcolor=blue,citecolor=blue,urlcolor=blue]{hyperref}
\numberwithin{equation}{section}
\theoremstyle{plain}
\newtheorem{theorem}{Theorem}[section]
\newtheorem{lemma}[theorem]{Lemma}
\newtheorem{proposition}[theorem]{Proposition}

\theoremstyle{definition}

\theoremstyle{remark}

\allowdisplaybreaks[4]
\newcommand{\dd}{\,\mathrm{d}}

\newcommand{\R}{\mathbb R}
\newcommand{\C}{\mathbb C}
\newcommand{\ind}{\operatorname{ind}}
\newcommand{\spann}{\operatorname{span}}

\title{Nondegeneracy and Morse Index of 
Ginzburg--Landau Vortices}

\begin{document}

\author[M. del Pino]{Manuel del Pino}
\address{Manuel del Pino, Department of Mathematical Sciences, University of Bath, Bath BA2 7AY, United Kingdom}
\email{mdp59@bath.ac.uk}

\author[Y. Liu]{Yong Liu}
\address{Yong Liu, School of Mathematics and Statistics, Beijing Technology and Business University, Beijing, China}
\email{yliumath@btbu.edu.cn} 

\author[M. Musso]{Monica Musso}
\address{Monica Musso, Department of Mathematical Sciences, University of Bath, Bath BA2 7AY, United Kingdom}
\email{mm2683@bath.ac.uk}

\author[J.C. Wei]{Juncheng Wei}
\address{Juncheng Wei, Department of Mathematics, Chinese University of Hong Kong, Shatin, New Territories, Hong Kong}
\email{wei@math.cuhk.edu.hk}

\author[W. Yang]{Wen Yang}
\address{Wen Yang, Department of Mathematics, Faculty of Science, University of Macau, Taipa, Macau, China}
\email{wenyang@um.edu.mo}

\begin{abstract}
We prove that the standard degree-two and degree-three vortex solutions of the Ginzburg-Landau equation are nondegenerate. Their Morse indices are also computed. The proof relies on new explicit upper and lower bounds of the modulus of these solutions and a comparison argument. It is expected that our method can be generalized to study higher degree solutions.
\end{abstract}

\maketitle
\section{Introduction and statement of the main results}

The Ginzburg--Landau (GL) model was introduced in the theory of
superconductivity \cite{GinzburgLandau}, and Abrikosov's work made
quantized vortices one of its basic objects \cite{Abrikosov}. Explicitly, the GL equation reads \begin{equation}\label{gl}
 -\Delta u=(1-|u|^2)u\qquad\hbox{in }\R^2.
\end{equation} Its energy functional is
$$
 E(u)=\int_{\R^2} \left\{\frac12|\nabla u|^2
 +\frac14(1-|u|^2)^2\right\}\dd x,
 \qquad u:\R^2 \longrightarrow\C.
$$
The GL equation is also the stationary equation for Gross--Pitaevskii dynamics. Solutions of \eqref{gl} are used as local models in many different problems. 

The theory developed in
\cite{BBH1993,BBH,Jerrard1999,JerrardSoner,Sandier1998} explains how
degree is quantized and how several vortex cores interact through a
renormalized energy.  The magnetic model and its asymptotic regimes are
treated systematically in \cite{SandierSerfaty,SandierSerfaty2008}.
Entire solutions and quantization results may be found in
\cite{BrezisMerleRiviere}.

For every positive integer $n$, there is a distinguished equivariant
solution
\begin{equation*}
 u_n(r,\theta)=f_n(r)e^{in\theta}.
\end{equation*}
It corresponds to the unique solution of
$$
 f_n''+\frac1r f_n'-\frac{n^2}{r^2}f_n+(1-f_n^2)f_n=0,
 \qquad f_n(0)=0,\qquad f_n(+\infty)=1,
$$
which is nonnegative and strictly increasing.  Existence, uniqueness, and
 qualitative properties of this   problem are classical,
see, for instance, \cite{BergerChen,ChenElliottQi,HerveHerve}.  It is known that $f_n$ has the following asymptotic behavior near the origin and infinity:
\begin{equation}\label{asym}
f_n(r)=
\begin{cases}
\kappa_n r^n\bigl(1+O(r^2)\bigr),~ r \to 0,\\
\\
1-\frac{n^2}{2r^2}
 -\frac{n^2(n^2+8)}{8r^4}-\frac{n^2(n^4+32n^2+128)}{16r^6}+O(r^{-8}),~ r\to\infty,
\end{cases}
\end{equation}
where $\kappa_n>0$. 

This paper investigates the spectral properties of vortex solutions $u_n$. The main difficulty is that no explicit formula is known for the vortex profile \(f_n\). Consequently, unlike in several integrable models, the linearized operator cannot be analyzed directly. The central idea of this paper is to replace \(f_n\) by explicit algebraic barriers which are sufficiently sharp to preserve the spectral information relevant to nondegeneracy and Morse index.

The
linearization at $u_n$ is
$$
 \mathcal L_nv=-\Delta v-(1-f_n^2)v
 +2u_n\operatorname{Re}(\overline{u_n}v),
$$
and the corresponding  quadratic form is given by
\begin{equation}\label{defIn}
I_n [v]=\int_{\R^2}\left(
 |\nabla v|^2-(1-f_n^2)|v|^2
 +2\bigl[\operatorname{Re}(e^{-in\theta}f_nv)\bigr]^2
 \right)\dd x.
\end{equation}
Phase and translation invariance yield three bounded Jacobi fields,
$
 iu_n,\partial_{x_1}u_n,\partial_{x_2}u_n.
$
None of them lies in $L^2(\R^2)$. Indeed, the phase field approaches modulus
one, and the translations have size $n/r+O(r^{-3})$.  
 
For the degree-one vortex, it is well known that it is stable (as a matter of fact, it is an energy minimizer).  Mironescu proved the positivity of its
second variation \cite{Mironescu}, and del Pino, Felmer, and Kowalczyk
gave Hardy-type formulation of minimality and nondegeneracy
\cite{delPinoFelmerKowalczyk}.  More recent work gives precise
solvability estimates for the degree-one linearized equation
\cite{delPinoJunemanMusso} and nonlinear coercivity in a renormalized
energy space \cite{GravejatPacherieSmets}.  Fourier analysis of the same
object also plays a role in anisotropic and coupled systems
\cite{DuanYang,LamyZuniga}.  In evolutionary problems, the threshold
part of the linearized spectrum is needed for dispersive and modulation
estimates; this viewpoint is developed, for instance, in
\cite{ColGer, LuhSch, PalaciosPusateri}.

For $n\geq2$, the situation turns out to be much more complicated.  A multiple vortex can decrease its
energy by separating into vortices of smaller degree, and the equivariant
solution is not a minimizer of the energy functional any more.  Instability in related radial and magnetic
problems was studied in \cite{Guo1996,GustafsonSigal}; see also
\cite{JaffeTaubes,OvchinnikovSigal}. In this class of non-minimizers, a natural question is the uniquness of solution. This question could be very difficult in general. However, locally around the standard vortex solution, the uniqueness should correspond to  nondegeneracy. Recently, for the complex sine-Gordon II equation, it is proved in \cite{Chenliu} that the corresponding degree-two vortex solution is actually degenerated and has nontrivial kernel, moreover, a family of non-radial solutions with explicit expressions is found. Here we would like to analyze the nondegeneracy property and exact number
of negative eigenvalues for the Ginzburg-Landau equation. Our result indicates that at least locally the radial solution is unique, whereas for degree-two and degree-three solutions, the Morse index should be consistent with that of the complex sine-Gordon II equation (In view of the number of free parameters of the Adler-Moser polynomials, we conjecture that for this equation the degree $n$ solution has Morse index $n(n-1)$, for all $n>0$).

Nondegeneracy and Morse index of vortex solutions play important roles in the study of questions realted to Ginzburg-Landau equation.  Suppose a problem has a
small parameter and a solution is built by placing a rescaled vortex
near a prescribed point.  After the leading error is removed, one has
to invert $\mathcal L_n$ on the complement of its geometric kernel.
An extra bounded Jacobi field would add a solvability condition and a
new modulation parameter. Nondegeneracy would imply that  no such parameter is hidden in the core problem. We also point out that nondegeneracy appears in the context of constructions of vortices on bounded
domains, magnetic vortices, and higher-dimensional vortex filaments using
a planar entire solution in the inner region and a global field in the
outer region.  The kernel determines which matching equations must be
solved.  This mechanism is visible in
\cite{BaraketPacard,delPinoKowalczykFilaments,
delPinoKowalczykMusso,PacardRiviere}.
Note that the translation fields are not in $L^2$, but
they are precisely the modes that move the solution.
In higher dimensions, concentration on codimension-two sets and the
resulting geometric motion are studied in
\cite{BethuelBrezisOrlandi,BethuelOrlandiSmets2006}.

The Morse index carries different information.  In a min-max
construction it gives the dimension of the energy descending space.  In a
parabolic or Hamiltonian evolution it identifies the number of unstable
spectral channels before neutral symmetries are modulated away.  It can
also be used to distinguish solution branches when a symmetric multiple
vortex bifurcates toward configurations of separated cores.  The
vortex-motion results in
\cite{BethuelOrlandiSmets2007,Lin1996,Lin1998,SandierSerfaty2004}
belong to other asymptotic regimes, but they illustrate why separating
core translations from genuine unstable directions matters.

\medskip
 The main difficulty is that, unlike the degree-one vortex, the higher-degree solutions are not minimizers of the Ginzburg--Landau energy. Consequently, neither positivity nor nondegeneracy follows from variational arguments, and a different approach is required.

 \medskip 
 Our main result in this paper is the following, which computes the Morse index and proves nondegeneracy.

\begin{theorem}\label{t:main}
For $n=2,3$, the standard degree $n$ vortex solutions are nondegenerated in the sense that
\begin{equation}\label{bker}
 \ker_{C^2(\R^2;\C)\cap L^\infty(\R^2)}\mathcal L_n
 =\spann_{\R}\{iu_n,\partial_{x_1}u_n,\partial_{x_2}u_n\}.
\end{equation}
In particular, the $L^2$ kernel is trivial:
\begin{equation}\label{l2ker}
 \ker_{L^2(\R^2)}\mathcal L_n=\{0\}.
\end{equation}
Moreover, the Morse index satisfies
$$
 \ind_{\R}(\mathcal L_2)=2,
 \qquad \ind_{\R}(\mathcal L_3)=6.
$$
\end{theorem}

The proof proceeds in three main steps. We first derive explicit global upper and lower bounds for the vortex profiles. We then identify the neutral Fourier modes associated with the symmetries of the equation. Finally, the remaining modes are reduced to one-dimensional Schrödinger operators whose positivity is established through suitable supersolutions and comparison arguments.

\medskip  
Since the standard vortex is radially symmetric, the linearized operator commutes with rotations. It is therefore natural to decompose perturbations into angular Fourier modes, reducing the spectral problem to a family of independent Fourier blocks, each represented by a coupled radial system.

\medskip  
For later reference, we also list the blockwise form of the index computation in the following result. 

\begin{theorem}\label{t:block}
Let $Q_m^{(n)}$ be the real radial copy of the $m$-th Fourier block defined by (\ref{qm}) for the degree $n$ vortex solution. Then 
for $n=2$,
\begin{equation}\label{tab2}
 \ind Q_2^{(2)}=1,
 \qquad \ker_{L^2}Q_2^{(2)}=\{0\}.
\end{equation}
For $n=3$,
\begin{equation}\label{tab3}
 \ind Q_m^{(3)}=1,
 \qquad \ker_{L^2}Q_m^{(3)}=\{0\},
 \qquad m=2,3,4.
\end{equation}
Moreover, all $m\geq2n-1$ blocks are strictly positive. 
\end{theorem}

It follows at once that the full real indices of the original linearized operator are two and six. 

A second difficulty arises because, for larger values of $n$, some Fourier blocks are expected to contain more than one negative eigenvalue. In that case, the construction of suitable supersolutions becomes considerably more involved.

We conjecture that, in view of the reflection comparison between different Fourier blocks (see Section~3), the Morse index is bounded above by $n(n-1)$ for every $n$. We also expect that the method developed in this paper can be extended to prove nondegeneracy for higher-degree vortex solutions.

\medskip
In \cite{beau2020}, nondegeneracy is claimed for the degree-$n$ vortex solution for every $n$. The proof of the corresponding general nondegeneracy result appears to rely on a variational argument requiring additional justification, since the admissibility of the test function used in the proof of Theorem~1.8 is not justified. The present paper provides a complete proof for the cases $n=2$ and $n=3$. To the best of our knowledge, these are the first rigorous and complete nondegeneracy results for higher-degree Ginzburg--Landau vortices. Numerical evidence of nondegeneracy can be found in \cite{Bara2002}.

\medskip
   Our proof proceeds in three stages. First we establish explicit global upper and lower bounds for the radial profile. Next we identify the neutral Fourier modes generated by the symmetries. 
   
   Finally, the remaining Fourier blocks are reduced to scalar one-dimensional Schrödinger operators, whose positivity is established by constructing suitable supersolutions.

 Our method also yields explicit global information about the vortex profiles $f_2$ and $f_3$, which may be of independent interest. Such bounds can be reused whenever a proof needs a sign in
the transition region rather than only the asymptotic series.  They are
simple algebraic functions, and their validity is proved by a maximum
principle.  These bounds are indeed inspired by the corresponding results of the complex sine-Gordon II equation, for which 
the standard degree-1 and degree-2 vortex solutions take the form (see, for instance, \cite{Chenliu}) 
$$
\Psi_1=\sqrt{Q_1}e^{i \theta}\quad\text{and} \quad\Psi_2=\sqrt{Q_2}e^{2i \theta},
$$
where 
$$
Q_{1}=\frac{r^{2}}{r^{2}+4}\quad\text{and} \quad Q_{2}=\frac{r^{4}\left(  r^{2}+24\right)  ^{2}}{r^{8}+64r^{6}+1152r^{4}+9216r^{2}+36864}.
$$ 
Higher degree vortex solutions also have explicit formulas, thanks to the integrable structure of this equation. This makes the analysis of complex sine-Gordon II equation easier than the Ginzburg-Landau equation. We will return to this equation in a future discussion.

The paper is organized as follows. Section~\ref{s:pre} first recalls the quadratic form and the real Fourier
blocks, then proves explicit lower and upper bounds for $f_2$ and $f_3$.  Section~\ref{s:proof} proves the two main theorems mentioned above.  The phase and translation
fields yield exact ground-state identities for $m=0,1$, while a  comparison treats high $m$.  For the remaining intermediate blocks we
replace $f_n$ by its lower algebraic barrier and analyze the corresponding linearized operator in those blocks.

\section{Explicit upper and lower bounds for
\texorpdfstring{$f_2$ and $f_3$}{f2 and f3}}
\label{s:pre}

Linear theory for the degree-one vortex solution has already been studied in detail in Chapter 3 of \cite{PacardRiviere}. 

An important observation is that it is more convenient to work with the conjugate form of the linearized operator. An equivalent formulation is obtained through the following Fourier block decomposition.

Following \cite[Chapter~3]{PacardRiviere}, it is convenient to work with the conjugate form of the linearized operator. This leads naturally to the following Fourier block decomposition.

Following \cite[Chapter~3]{PacardRiviere}, it is convenient to work with the conjugate form of the linearized operator. We first derive the corresponding Fourier block decomposition. Fix $m\ge1$ and set

\begin{align*}
 v_m^{\mathrm c}&=a(r)e^{i(n-m)\theta}
                  +b(r)e^{i(n+m)\theta},\\
 v_m^{\mathrm s}&=ia(r)e^{i(n-m)\theta}
                  -ib(r)e^{i(n+m)\theta},
\end{align*}
where $a,b$ are real.  Since
\begin{align*}
 \operatorname{Re}(e^{-in\theta}v_m^{\mathrm c})&=(a+b)\cos(m\theta),\\
 \operatorname{Re}(e^{-in\theta}v_m^{\mathrm s})&=(a+b)\sin(m\theta),
\end{align*}
orthogonality of sines and cosines tells us that
\begin{equation*}
I_n [v_m^{\mathrm c}]
 =I_n [v_m^{\mathrm s}]=2\pi Q_m^{(n)}[a,b],
\end{equation*}
with $I_n$ defined in \eqref{defIn}, where 
\begin{equation}\label{qm}
\begin{split}
 Q_m^{(n)}[a,b]=\int_0^\infty\bigg\{&a_r^2+b_r^2
 +\left(\frac{(n-m)^2}{r^2}+2f_n^2-1\right)a^2\\
 &+\left(\frac{(n+m)^2}{r^2}+2f_n^2-1\right)b^2
 +2f_n^2ab\bigg\}r\dd r
\end{split}
\end{equation} defined on natural energy space.
Different values of $m$, and the two copies for a fixed $m$, are
orthogonal both for the real $L^2$ product and for the Hessian form.
Conversely, grouping the Fourier coefficients with angular momenta
$n-m$ and $n+m$ gives the orthogonal real decomposition of every smooth
compactly supported perturbation. 

 We use
\begin{equation*}
 \nu(Q)=\sup\{\dim S:Q\leq0\text{ on the linear space }S\}.
\end{equation*}
Thus $\nu$ is the negative index (denoted by $\ind$) plus the $L^2$ nullity.

\begin{lemma}
For every $m\geq1$, the operator associated with
$Q_m^{(n)}$ has essential spectrum $[0,\infty)$.  Its spectrum below
zero consists of finitely many eigenvalues of finite multiplicity.  In
particular,
\begin{equation}\label{nunull}
 \nu(Q_m^{(n)})=\ind Q_m^{(n)}+\dim\ker_{L^2}Q_m^{(n)}.
\end{equation}
\end{lemma}

\begin{proof}
Use the constant orthogonal variables
$A=(a+b)/\sqrt2$ and $P=(a-b)/\sqrt2$.  From
\eqref{asym}, the potential matrix at infinity
has the form
\begin{equation*}
 \begin{pmatrix}2&0\\0&0\end{pmatrix}
 +\frac1{r^2}
 \begin{pmatrix}m^2-2n^2&-2nm\\-2nm&m^2\end{pmatrix}
 +O(r^{-4}),
\end{equation*}
up to interchanging the labels $A,P$. 
Since multiplication by a
bounded matrix-valued function converging to zero at infinity is
form-compact relative to the direct sum of the radial operators
\[
-\partial_r^2-r^{-1}\partial_r+2,
\qquad
-\partial_r^2-r^{-1}\partial_r
\]
acting in $L^2((0,\infty),r\,dr)$, Weyl's theorem implies that the
essential spectrum is
\[
\sigma_{\rm ess}(Q_m^{(n)})=[0,\infty).
\]
To prove that the negative spectrum is finite, it is enough to show
that the quadratic form is nonnegative outside a compact interval.
Indeed, for $r\ge R$, we estimate the coupling term using
$2ab\le a^2+b^2$:
\[
\frac{4|nm|}{r^2}|AP|
\le
A^2+\frac{4n^2m^2}{r^4}P^2.
\]
Choosing $R$ sufficiently large, we have
\[
\frac{m^2}{r^2}
-\frac{4n^2m^2}{r^4}
=
\frac{m^2}{r^2}
\left(1-\frac{4n^2}{r^2}\right)
\ge
\frac{m^2}{2r^2},
\]
while the massive coefficient remains bounded below by $1$. Hence the
quadratic form is nonnegative on $(R,\infty)$.

Splitting the form into the interior and exterior regions, the exterior
contributes no negative spectrum. On the bounded interval $(0,R)$, the
regular matrix Sturm--Liouville operator (for instance with Dirichlet
boundary condition at $r=R$) has compact resolvent, and therefore only
finitely many negative eigenvalues. By the minimax principle, the full
operator likewise has only finitely many negative eigenvalues.

Since the negative spectrum consists only of isolated eigenvalues of
finite multiplicity, the spectral theorem yields
\[
\nu(Q_m^{(n)})
=
\ind Q_m^{(n)}
+
\dim\ker_{L^2}Q_m^{(n)},
\]
which is precisely \eqref{nunull}.
\end{proof}

We shall also analyze the asymptotic behavior of the solutions near the origin and infinity.  The argument here follows the analysis carried out for the degree-one
vortex in \cite[Section~3.3, Propositions~3.1 and~3.2]{PacardRiviere}. The difference  is that we give the estimates for general degree $n$.

Near the origin,
\[
 f_n(r)=\kappa_n r^n
 \left(1-\frac{r^2}{4(n+1)}+O(r^4)\right),
 \qquad f_n^2(r)=\kappa_n^2r^{2n}(1+O(r^2)).
\]
The term involving $b$ in the equation for $a$ is $f_n^2b$, and the
term involving $a$ in the equation for $b$ is $f_n^2a$.  Keeping only
the dominant terms near the origin gives
\[
 -a''-\frac1r a'+\frac{(n-m)^2}{r^2}a=0,
 \qquad
 -b''-\frac1r b'+\frac{(n+m)^2}{r^2}b=0.
\]
If $m\ne n$, the first equation has the two solutions
$r^{|n-m|}$ and $r^{-|n-m|}$; if $m=n$, they are $1$ and $\log r$.
The two solutions of the second equation are $r^{n+m}$ and
$r^{-(n+m)}$.  Regularity at the origin excludes the singular solutions
$r^{-|n-m|}$, $\log r$, and $r^{-(n+m)}$.  Therefore
\begin{equation}\label{orgb}
 a(r)=O(r^{|n-m|}),\qquad b(r)=O(r^{n+m})\qquad (r\to0).
\end{equation}

At infinity, set $A=a+b$ and $P=a-b$.  Since
$f_n^2=1-n^2r^{-2}+O(r^{-4})$, their equations take the asymptotic form
\begin{align*}
 -A''-\frac1rA'+\bigl(2+O(r^{-2})\bigr)A
 -\frac{2nm}{r^2}P&=0,\\
 -P''-\frac1rP'+\left(\frac{m^2}{r^2}+O(r^{-4})\right)P
 -\frac{2nm}{r^2}A&=0.
\end{align*}
For $m\geq2$, 
with the following behavior:
\begin{align*}
 P_1&=r^m(1+O(r^{-2})),
 &A_1&=nm\,r^{m-2}(1+O(r^{-2})),\\
 P_2&=r^{-m}(1+O(r^{-2})),
 &A_2&=nm\,r^{-m-2}(1+O(r^{-2})),\\
 A_3&=r^{-1/2}e^{\sqrt2r}(1+O(r^{-1})),
 &P_3&=O(r^{-2}A_3),\\
 A_4&=r^{-1/2}e^{-\sqrt2r}(1+O(r^{-1})),
 &P_4&=O(r^{-2}A_4).
\end{align*}
Every solution is a linear combination of these four solutions.  If
$a$ and $b$ are bounded, then $P=a-b$ and $A=a+b$ are bounded. Hence
the coefficients of $(A_1,P_1)$ and $(A_3,P_3)$ vanish.  It follows
that
\begin{equation}\label{bdecay}
 P=O(r^{-m}),\qquad
 A=O(r^{-m-2})+O(r^{-1/2}e^{-\sqrt2 r})\qquad (r\to\infty).
\end{equation}
Differentiating the equations, or equivalently the asymptotic
expansions above, yields
\[
 a'(r)=
 \begin{cases}
  O(r^{|n-m|-1}),&m\ne n,\\
  O(r),&m=n,
 \end{cases}
 \qquad b'(r)=O(r^{n+m-1})\qquad (r\to0),
\]
and
\[
 P'=O(r^{-m-1}),\qquad
 A'=O(r^{-m-3})+O(r^{-1/2}e^{-\sqrt2 r})
 \qquad (r\to\infty).
\]
Consequently, for every bounded solution that is smooth at the origin
and has $m\geq2$,
\[
 \int_0^\infty\bigl(a^2+b^2+(a')^2+(b')^2\bigr)r\,\dd r<\infty.
\]

\bigskip
One of the main purposes in this section is to establish suitable lower and upper bounds for the standard vortex solutions. The reason we want to do this is the following. The linearized operator is associated to the vortex solution which does not have explicit formulas. This means that we actually don't have much information on this operator. In general, an operator can have kernels if there are no ``symmetry'' properties are imposed on it. For the degree one vortex solution, we already know that it is a minimizer of the energy functional. This is exactly the symmetry property used to prove its nondegeneracy. To prove the nondegeneracy of higher degree vortex solution, then it seems to be necessary to obtain precise bounds. The rest of this section is devoted to this aim.

The next lemma provides the comparison principle that underlies the construction of explicit algebraic barriers for $f_n$. Once an approximate profile satisfies a suitable differential inequality and the correct asymptotic ordering at infinity, it becomes a genuine global upper or lower bound. For this purpose, for a
positive function $g$, let us define
\[
 \mathcal E_n(g):=g''+\frac1r g'-\frac{n^2}{r^2}g+(1-g^2)g.
\]

\begin{lemma}\label{l:ratio}
Assume that $g(r)=\gamma r^n(1+O(r^2))$ near zero, with $\gamma>0$,
and that $g(r)\to1$ as $r\to\infty$.  If
$\mathcal E_n(g)\leq0$ and $f_n/g<1$ near infinity, then $f_n<g$ on
$(0,\infty)$.  If $\mathcal E_n(g)\geq0$ and $f_n/g>1$ near infinity,
then $f_n>g$ on $(0,\infty)$.  
\end{lemma}

\begin{proof}
Write $f_n=zg$.  Subtracting the equations for $f_n$ and $g$ gives
\begin{equation}\label{ratio}
 (rg^2z')'=rg^2z\left[g^2(z^2-1)-\frac{\mathcal E_n(g)}g\right].
\end{equation}
Since both $f_n$ and $g$ satisfy
\[
f_n(r),\,g(r)=\gamma r^n(1+O(r^2)),
\]
the ratio $z=f_n/g$ extends continuously to $r=0$,
and
\[
rg^2z'(r)\to0
\qquad(r\to0).
\]
Suppose
first that $\mathcal E_n(g)\leq0$.  On a component of $\{z>1\}$ the
right side of \eqref{ratio} is positive.  
At the left endpoint of such a component,
\[
z=1,\qquad z'\ge0,
\]
while, if the component starts at the origin, we have
\[
rg^2z'=0.
\]
Indeed, since $(rg^2z')'>0$ throughout the component,
the quantity $rg^2z'$ is strictly increasing.
Hence it remains positive, implying $z'>0$ everywhere on the
component. Consequently $z$ is strictly increasing and cannot
return to the value $1$ at its right endpoint.
  This contradicts $z<1$ near infinity.
Thus $z\le1$ throughout $(0,\infty)$. Since $z\not\equiv1$
(otherwise $\mathcal E_n(g)\equiv0$ and $g=f_n$), the strong
maximum principle applied to \eqref{ratio} yields
\[
z<1
\quad\text{on }(0,\infty).
\]  The second assertion
follows in the same way from a component of $\{z<1\}$.
\end{proof}

To simplify notations, let us now define
\begin{equation}\label{g3}
 g_C(r)=\frac{r^3}{\sqrt{r^6+9r^4+99r^2+C}},
 \qquad C>0.
\end{equation}
The main result of this section is the following
\begin{proposition}\label{p:bd3}
For every $r>0$, there holds $$
 \frac{r^2}{\sqrt{r^4+4r^2+30+6\sqrt{21}}}
 < f_2(r) <
 \frac{r^2}{\sqrt{r^4+4r^2+24}}.
$$
We also have \begin{equation}\label{bd3}
 g_{4000}(r)<f_3(r) <g_C(r),\quad \text{for} \quad 0<C<792.
\end{equation}
\end{proposition}

\begin{proof}
Since the proof for $f_2$ and $f_3$ is similar, 
we focus our main attention on the case of $f_3$. For $f_2$, we simply use the fact that $$\frac{\mathcal E_2(\phi_C)}{\phi_C}
 =
 \frac{(C-24)(r^4+C)-8(C+6)r^2}
 {(r^4+4 r^2+C)^2},$$
where $$\phi_C(r)=\frac{r^2}{\sqrt{r^4+4r^2+C}},\qquad C>0.$$ 

Now in the case of $f_3$, for the function in \eqref{g3}, put
$x=r^2$ and $D_C(x)=x^3+9x^2+99x+C$.  Direct substitution gives
\begin{equation*}
 \frac{\mathcal E_3(g_C)}{g_C}=\frac{B_C(x)}{D_C(x)^2},
\end{equation*}
where
\begin{equation}\label{bpoly}
 B_C(x)=(C-1440)x^3-(18C+4455)x^2
 +(18C-49005)x+C(C-792).
\end{equation}
When $0<C<792$, every coefficient in \eqref{bpoly} is negative, so
$\mathcal E_3(g_C)<0$.  For $C=4000$,
\[
 B_{4000}(x)=2560x^3-76455x^2+22995x+12832000.
\]
For $x\geq0$,
\[
 2560x^3-76455x^2
 \geq-\frac{4\cdot76455^3}{27\cdot2560^2}.
\]
A direct computation gives
\[
 12832000\cdot27\cdot2560^2-4\cdot76455^3
 =482956326814500>0
\]
therefore shows that $B_{4000}(x)>0$.

The asymptotic expansions needed to apply Lemma~\ref{l:ratio} are
\begin{align*}
 f_3(r)&=1-\frac9{2r^2}-\frac{153}{8r^4}
 -\frac{4473}{16r^6}+O(r^{-8}),\\
 g_C(r)&=1-\frac9{2r^2}-\frac{153}{8r^4}
 +\left(\frac{7047}{16}-\frac C2\right)r^{-6}+O(r^{-8}).
\end{align*}
For $C<792$, the coefficient of $r^{-6}$ in the expansion of $g_C$
is strictly larger than that of $f_3$, and hence
\[
g_C(r)-f_3(r)>0
\]
for all sufficiently large $r$. Equivalently,
\[
\frac{f_3}{g_C}<1.
\]
Similarly, when $C=4000$,
\[
g_{4000}(r)-f_3(r)<0,
\]
so that
\[
\frac{f_3}{g_{4000}}>1
\]
for sufficiently large $r$.
 Hence $f_3/g_C<1$ in the
first case and $f_3/g_{4000}>1$ in the second case for all sufficiently
large $r$.  Lemma~\ref{l:ratio} completes the proof.
\end{proof}

We remark that for complex sine-Gordon II equation, the polynomial is part of an exact solution generated by
the integrable B\"acklund--Schlesinger hierarchy.  In the polynomial
family under consideration, the degree-\(n\) solution involves a
polynomial of degree \(2n^2\) in \(r\), or degree \(n^2\) in
\(x=r^2\).  Thus the degree-one and degree-two solutions involve
degrees \(2\) and \(8\) in \(r\), respectively.

For Ginzburg-Landau equation, the functions used in our estimates are not exact
polynomial solutions.  They are comparison barriers of the form
\[
g_n(r)=\frac{r^n}{\sqrt{D_n(r^2)}}.
\]
Their degree is determined only by the required behavior at
\(r=0\) and \(r=\infty\). In particular, it is natural to choose $D_n$ with degree $n$. It is possible to obtain a general explicit formula for the upper and lower bound of $f_n$, but we will not pursue this here, since the analysis of Morse index and nondegeneracy requires quite precise estimate of $f_n$, but a general closed form expression of the lower and upper bound may not be sufficient for this purpose. It is also worth pointing out that for the simplest $n=1$ case, we can actually prove 
\[
\frac{r}{\sqrt{r^2+4}}
<f_1(r)<
\frac{r}{\sqrt{r^2+1}},
\qquad r>0.
\] It is also interesting to note that the left hand side is the precise form of the solution corresponding to the complex sine-Gordon II equation.

\section{Nondegeneracy and Morse index}
\label{s:proof}

In this section, we analyze the Fourier blocks of the linearized operator and prove the nondegeneracy, compute the Morse index. For degree-one vortex solution, the classical stability and nondegeneracy result is proved in \cite{Mironescu,PacardRiviere}, see also \cite{delPinoFelmerKowalczyk}.

To better understand the differences between degree one vortex solution and higher degree ones, let us now explain the main steps of the proof. We will analyze the linearized operator according to different Fourier modes. The mode $m=0$ is exceptional because the phase and
amplitude variables decouple, and the phase symmetry produces the only zero
mode in that block, this step is essentially same as the degree one vortex solution.  The mode $m=1$ contains the two translation fields and
admits an exact ground-state decomposition. Comparing a general block with this reference gives
\eqref{high}: the comparison is strictly positive for $m\geq2n-1$ (including
the endpoint $m=2n-1$), so all such blocks are positive, whereas for
$2\leq m\leq2n-2$ one centrifugal coefficient has the wrong sign and the
comparison alone is inconclusive.  These intermediate modes are consequently
the only ones in which negative spectrum or an additional kernel can occur.
For them, we replace $f_n$ by an explicit lower barrier $g$, use the logarithmic
variable $t=\log r$, and complete the square in one component of the resulting
two-component form.  Dropping that nonnegative square yields the scalar
comparison form $\ell_m$ in \eqref{slow}; non-negative supersolutions for its
Schr\"odinger operator bound the dimension of the nonpositive subspace, while
explicit negative test functions give the matching lower bounds for the Morse
index.

We now analyze these Fourier blocks in details. As we mentioned above, the modes \(m=0\) and \(m=1\) correspond to the symmetries of the equation and therefore should produce the geometric kernel. The remaining modes require a completely different argument.

Let us first of all treat the case of $m=0$, this is more or less standard. We can write $v=e^{in\theta}(\alpha+i\beta)$.
Applying the orthogonal change of variables
\[
\alpha=\frac{a+b}{\sqrt2},
\qquad
\beta=\frac{a-b}{\sqrt2},
\]
we obtain
\[
Q_0^{(n)}[a,b]=Q_0^+[\alpha]+Q_0^-[\beta],
\]
where
\begin{align*}
 Q_0^-[\beta]&=\int_0^\infty\left\{\beta_r^2+
 \left(\frac{n^2}{r^2}+f_n^2-1\right)\beta^2\right\}r\dd r,\\
 Q_0^+[\alpha]&=Q_0^-[\alpha]+2\int_0^\infty f_n^2\alpha^2r\dd r.
\end{align*} 
The differential operator associated with $Q_0^-$ annihilates the
positive function $f_n$. Consequently, the ground-state
representation gives
\begin{equation}\label{phase}
 Q_0^-[\beta]
 =
 \int_0^\infty f_n^2
 \left|
 \left(\frac{\beta}{f_n}\right)'
 \right|^2r\,\dd r.
\end{equation}
Indeed, for $\beta$ compactly supported in $(0,\infty)$, one expands
the square on $(\varepsilon,R)$ and uses the profile equation
\[
 f_n''+\frac1r f_n'-\frac{n^2}{r^2}f_n+(1-f_n^2)f_n=0
\]
to integrate the cross term. The regular behavior of $f_n$ and
$\beta$ at the origin, together with a cutoff at infinity, eliminates
the boundary terms. The identity then extends to the full form domain
by closure.

In particular, $Q_0^-\geq0$. The general solution of the corresponding
homogeneous equation is a linear combination of
\[
 f_n(r)
 \qquad\text{and}\qquad
 f_n(r)\int_c^r\frac{\dd s}{s f_n(s)^2}.
\]
The second solution behaves like $r^{-n}$ near the origin and like
$\log r$ at infinity. It is therefore singular at the origin and
unbounded at infinity. Hence the space of bounded solutions that are
regular at the origin is generated by
\(
 \mathbb R f_n,
\)
corresponding in the plane to the phase Jacobi field
\(
  i u_n.
\)
Finally, from \eqref{phase},
\[
 Q_0^+[\alpha]
 =
 \int_0^\infty f_n^2
 \left|
 \left(\frac{\alpha}{f_n}\right)'
 \right|^2r\,\dd r
 +
 2\int_0^\infty f_n^2\alpha^2r\,\dd r,
\]
and hence $Q_0^+[\alpha]>0$ for every nonzero $\alpha$ in its form
domain. Moreover, every bounded solution of the amplitude equation
that is regular at the origin decays exponentially at infinity and
therefore belongs to the form domain. Thus the amplitude block has no
nontrivial bounded kernel.

Next we treat the case of $m=1$. Set
\begin{equation}\label{pq}
 q(r)=f_n'+\frac{nf_n}{r},\qquad p(r)=\frac{nf_n}{r}-f_n'.
\end{equation}
Both functions are strictly positive. The positivity of $q$ follows
immediately from the positivity and monotonicity of $f_n$. To prove that
$p>0$, set
\[
 F(t)=f_n(e^t),
 \qquad
 s(t)=\frac{\dot F(t)}{F(t)}.
\]
The profile equation gives
\[
 \dot s
 =
 n^2-e^{2t}(1-F^2)-s^2.
\]
Moreover, the expansion of $f_n$ at the origin yields
\[
 s(t)
 =
 n-\frac{e^{2t}}{2(n+1)}+O(e^{4t})
 \qquad\text{as }t\to-\infty.
\]
Hence $s<n$ for all sufficiently negative $t$. If $s$ had a first
contact with $n$ from below, then at the contact point one would have
$\dot s\geq0$. On the other hand, the differential equation gives
\[
 \dot s=-e^{2t}(1-F^2)<0,
\]
since $0<F<1$. This is a contradiction. Therefore $s(t)<n$ for every
$t$, and hence
\[
 p(r)=\frac{f_n(r)}{r}\bigl(n-s(\log r)\bigr)>0.
\]
Differentiating $u_n$ with respect to $x_1$ gives
\begin{equation}\label{trans}
 2\partial_{x_1}u_n
 =q(r)e^{i(n-1)\theta}-p(r)e^{i(n+1)\theta}.
\end{equation}
Consequently, by \eqref{trans}, $q$ and $p$ solve
\begin{align}
 -q''-r^{-1}q'+\left((n-1)^2r^{-2}+2f_n^2-1\right)q-f_n^2p&=0,
 \label{qeq}\\
 -p''-r^{-1}p'+\left((n+1)^2r^{-2}+2f_n^2-1\right)p-f_n^2q&=0.
 \label{peq}
\end{align}
Applying the scalar ground-state representation separately to the
$q$- and $p$-diagonal terms in $Q_1^{(n)}[a,-c]$, and using
\eqref{qeq}--\eqref{peq}, we obtain
\begin{equation}\label{tground}
\begin{split}
 Q_1^{(n)}[a,-c]=\int_0^\infty\bigg\{&q^2
 \left|\left(\frac aq\right)'\right|^2+p^2
 \left|\left(\frac cp\right)'\right|^2\\
 &+f_n^2pq\left(\frac aq-\frac cp\right)^2\bigg\}r\dd r.
\end{split}
\end{equation}
Indeed, the two diagonal remainder terms are
\[
 f_n^2\frac pq\,a^2
 \qquad\text{and}\qquad
 f_n^2\frac qp\,c^2,
\]
whereas the original off-diagonal term is $-2f_n^2ac$. Their sum is
precisely
\[
 f_n^2pq
 \left(\frac aq-\frac cp\right)^2.
\]
Identity \eqref{tground} shows that $Q_1^{(n)}\geq0$.
Moreover, equality
in the form domain forces
\[
 \left(\frac aq\right)'=0,
 \qquad
 \left(\frac cp\right)'=0,
 \qquad
 \frac aq=\frac cp,
\]
and therefore
\[
 (a,-c)=\lambda(q,-p)
\]
for some $\lambda\in\mathbb R$.

The same conclusion holds for bounded solutions that are regular at the
origin. Indeed, applying \eqref{tground} to suitable cutoffs of such a
solution and using its asymptotic behavior at zero and infinity shows
that the cutoff errors tend to zero. Hence the bounded radial kernel of
the $m=1$ block is
\(
 \mathbb R(q,-p).
\)
The cosine and sine angular components of this radial mode correspond
to the two translational Jacobi fields \(
 \partial_{x_1}u_n\)
 and \(\partial_{x_2}u_n.
\)

Having analyzed the symmetry modes, we now turn to the genuinely
unstable Fourier blocks. Identity (\ref{tground}) immediately implies positivity of sufficiently high modes through the comparison 
\begin{equation}\label{high}
 Q_m^{(n)}-Q_1^{(n)}=\int_0^\infty\left\{
 \frac{(m-1)(m-2n+1)}{r^2}a^2+
 \frac{(m-1)(m+2n+1)}{r^2}b^2\right\}r\dd r.
\end{equation}
Combining \eqref{high} with the ground-state identity
\eqref{tground}, we conclude that
\(
Q_m^{(n)}>0
\)
for every nonzero pair whenever \(m\ge2n-1\).  For $m>2n-1$ both added
coefficients are positive. When \(m=2n-1\), the coefficient of \(a^2\) vanishes, whereas the
coefficient of \(b^2\) remains strictly positive. Since the
translation mode satisfies \(p>0\), it cannot lie in the kernel of
\(Q_{2n-1}^{(n)}\). Hence the comparison is still strict.

The previous analysis tells that only
\begin{equation*}
 2\leq m\leq2n-2
\end{equation*}
 can contribute negative spectrum or an additional bounded kernel. These intermediate modes constitute the main difficulty of the proof.

Fix an explicit lower barrier $g=g_{n,C_L}<f_n$ and let $Q_{m,g}^{(n)}$ denote
\eqref{qm} with $f_n$ replaced by $g$.  Since
$g<f_n$,
\begin{equation}\label{forder}
 Q_m^{(n)}[a,b]-Q_{m,g}^{(n)}[a,b]
 =2\int_0^\infty(f_n^2-g^2)(a^2+ab+b^2)r\dd r>0
\end{equation}
for every nontrivial pair \((a,b)\).  
Since
\[
Q_m^{(n)}\ge Q_{m,g}^{(n)}
\]
with strict inequality on every nonzero element, the minimax
principle implies
\begin{equation}\label{nuord}
\nu(Q_m^{(n)})\leq\nu(Q_{m,g}^{(n)}).
\end{equation}
Passing to logarithmic variables
\[
t=\log r,\qquad x=e^{2t},
\]
and writing
\[
a=-d,\qquad b=c,
\]
the quadratic form becomes
\begin{equation}\label{e:model}
\begin{aligned}
 Q_{m,g}^{(n)}[c,d]
 &=\int_\R\left\{\dot c^2+\dot d^2 \right\} \, dt \\
 &+\int_\R \left\{
 ((n+m)^2-P)c^2+((n-m)^2-P)d^2+W(c-d)^2\right\}\dd t.
 \end{aligned}
\end{equation} 

We now introduce the effective potentials
\begin{equation}\label{av}
 A_m(t)=(n+m)^2-P,\qquad
 V_m(t)=(n-m)^2-P+\frac{A_mW}{A_m+W}.
\end{equation}
For every model used below, $P(t)<(n+2)^2$, hence $A_m(t)>0$.   
Observe that
$$
 A_mc^2+W(c-d)^2=(A_m+W)
 \left(c-\frac{W}{A_m+W}d\right)^2+\frac{A_mW}{A_m+W}d^2.
$$

Note that since \(A_m>0\), the \(c\)-component in the quadratic form is coercive and cannot by itself generate an unstable direction. For each fixed \(d\), completing the square in \(c\) amounts to minimizing the zeroth-order energy \(A_mc^2+W(c-d)^2\), yielding the effective contribution \(A_mW(A_m+W)^{-1}d^2\). The remaining square and the kinetic term \(\dot c^{,2}\) are nonnegative and may therefore be discarded. 

Consequently, the two-component problem is reduced to the scalar Schrödinger form $\ell_m$, which will be analyzed in the remainder of the proof.

This reduces the coupled two-component problem to a scalar Sturm--Liouville form, to which oscillation and supersolution arguments can be applied. The following lemma makes this reduction precise.

\begin{lemma}
 There holds
\begin{equation}\label{slow}
 Q_{m,g}^{(n)}[c,d]\geq
 \ell_m[d]:=\int_\R(\dot d^2+V_md^2)\dd t.
\end{equation}   
\end{lemma}
We next rewrite the effective potential $V_m$ in a more convenient form. This representation reduces all subsequent positivity checks to explicit polynomial inequalities.

The two lower barriers have the same form
\[
 g(r)^2=\frac{x^n}{D(x)},\qquad x=r^2.
\]
Here $D$ is simply the denominator of the barrier: For $n=2$, by the result of the previous section,  we can use $g=\phi_{C_*},$  where
\[
 \phi_C(r)=\frac{r^2}{\sqrt{r^4+4r^2+C}},\qquad
 C_*=30+6\sqrt{21},\qquad D(x)=x^2+4x+C_*,
\]
whereas for $n=3$ we use $g=g_{4000}$ from \eqref{g3}, so
\[
 D(x)=x^3+9x^2+99x+4000.
\]
It follows at once that
\[
 W=\frac{x^{n+1}}{D(x)},\qquad P=\frac{p_0(x)}{D(x)} ,\qquad
 p_0(x):=x\bigl(D(x)-x^n\bigr).
\]
Thus \(p_0\) is precisely the numerator of \(P=x(1-g^2)\).  In the same way,
\[
 A_m=\frac{(n+m)^2D-p_0}{D},\qquad
 A_m+W=\frac{E_m}{D},
\]
where
\begin{equation}\label{spoly}
 E_m(x):=(n+m)^2D(x)-p_0(x)+x^{n+1}.
\end{equation}
Finally, define the numerator of $V_m$ by
\begin{equation}\label{npoly}
\begin{aligned}
 N_m(x)&:=\bigl((n-m)^2D(x)-p_0(x)\bigr)E_m(x)
 \\
 &+\bigl((n+m)^2D(x)-p_0(x)\bigr)x^{n+1}.
 \end{aligned}
\end{equation}
Substitution in \eqref{av} now gives the identity
\[
 V_m(t)=\frac{N_m(x)}{D(x)E_m(x)},\qquad x=e^{2t}.
\]
Since $\frac{s}{DE_2}>0$ for all $s\in(0,1)$, the sign of
\[
(-\partial_t^2+V_2)h_-
\]
is determined by the explicit polynomial \(N_m\).

The following proposition provides the key upper bound for \(\nu(Q_m^{(n)})\). Corresponding lower bounds will later be obtained by constructing explicit test functions for which the quadratic form is negative.

\medskip
 Before stating the proposition, let us briefly explain the underlying idea, which is closely related to the Sturm--Courant nodal principle.
 The points \(t_1,\ldots,t_k\) divide \(\mathbb R\) into \(k+1\) intervals. A positive supersolution on each interval, together with the ground-state identity, implies that \(\ell_m[\eta]>0\) whenever \(\eta\neq0\) vanishes at all the points \(t_j\). Since imposing the \(k\) conditions
\(
\eta(t_1)=\cdots=\eta(t_k)=0
\)
reduces the dimension by at most \(k\), the scalar form \(\ell_m\) can have at most \(k\) nonpositive directions.

The comparison
\(
Q_{m,g}^{(n)}[c,d]\ge \ell_m[d]
\)
allows the scalar estimate to be transferred to the original
two-component form. Indeed, the projection
\(
(c,d)\longmapsto d
\)
is injective on every nonpositive subspace of $Q_{m,g}^{(n)}$, since
\[
Q_{m,g}^{(n)}[c,0]
=
\int_{\mathbb R}
\bigl(\dot c^{\,2}+(A_m+W)c^2\bigr)\,dt>0
\]
whenever $c\not\equiv0$. Consequently, the dimension of any
nonpositive subspace of $Q_{m,g}^{(n)}$ is bounded by that of the
corresponding nonpositive subspace of $\ell_m$. Since each splitting
point contributes at most one nonpositive direction, we obtain
\[
\nu(Q_m^{(n)})\le k.
\]
We emphasize that this argument is intrinsically one-dimensional: it relies on the Sturm oscillation mechanism and has no direct analogue in higher dimensions.

\begin{proposition}\label{p:trace}
Suppose there are points $t_1<\cdots<t_k$ and, on every component of
$\R\setminus\{t_1,\ldots,t_k\}$, a positive function $h$ satisfying
$(-\partial_t^2+V_m)h>0$.  
Then the following implication holds for every
$\eta\in H^1_{\rm loc}(\R)$ with
$\ell_m[\eta]<\infty$:
\begin{equation}\label{trace}
 \eta\ne0,\quad \eta(t_1)=\cdots=\eta(t_k)=0
 \quad\Longrightarrow\quad \ell_m[\eta]>0,
\end{equation}
and $\nu(Q_m^{(n)})\leq k$.
\end{proposition}

\begin{proof} We first establish the implication \eqref{trace}. Note that $\eta$
  is continuous, so the values $  \eta(t_j)  $ are well-defined, and the evaluation maps $  \eta\mapsto\eta(t_j)  $ are continuous linear functionals on any finite-dimensional subspace. The potentials $  V_m  $ under consideration are smooth and bounded from below, with finite limits at infinity; in particular finite energy forces $  \dot\eta\in L^2(\R)  $. Let $  I  $ be a connected component of $  \R\setminus\{t_1,\dots,t_k\}  $ and let $  h>0  $ on $  I  $ satisfy $  (-\partial_t^2+V_m)h>0  $. For $  \eta  $ smooth and compactly supported in the interior of $  I  $, expanding $  h^2|(\eta/h)^\cdot|^2  $ and integrating by parts gives 
\begin{equation}
\label{e:gscalar}
\int_I(\dot\eta^2+V_m\eta^2)\,dt =\int_I h^2\Bigl|\Bigl(\frac\eta h\Bigr)^\cdot\Bigr|^2\,dt +\int_I\frac{(-\partial_t^2+V_m)h}h\,\eta^2\,dt. 
\end{equation}
The first integrand on the right-hand side is nonnegative, whereas the second is strictly positive wherever $  \eta\not\equiv0  $. If $  \eta  $ merely lies in the energy class and vanishes at the finite ends of $  I  $, the elementary bound 
$$|\eta(t)|^2\le|t-t_j|\int_{t_j}^t|\dot\eta|^2$$
shows that truncations supported away from those ends converge to $  \eta  $ in energy. Passing to the limit in \eqref{e:gscalar} yields $  \ell_m[\eta]\ge0  $, with equality only if both integrands on the right vanish a.e. If $  I  $ is unbounded one must also cut off at infinity. Finite energy does not give $  \eta\in L^2  $ on a tail where $  V_m\to0  $, nor does $  \dot\eta\in L^2  $ force a finite limit of $  \eta  $. What it does give is the growth control 
$$\frac{\eta(t)^2}{|t|}\to0\quad\mbox{as}\quad |t|\to\infty.$$ 
In fact, as $  t\to+\infty  $, fix $  \varepsilon>0  $ and choose $  T  $ with 
$$  \|\dot\eta\|_{L^2(T,\infty)}^2<\varepsilon/2  ,$$ 
then 
$$\frac{\eta(t)^2}{t}\le\frac{2\eta(T)^2}{t}+2\|\dot\eta\|_{L^2(T,\infty)}^2,$$ 
so it is at most $  \varepsilon  $. Let $  \chi_L  $ be a Lipschitz cut-off equal to $  1  $ on $  I\cap[-L,L]  $ and affine of slope $  \pm1/L  $ on transition intervals $  T_L  $ of total length $  O(L)  $. Then $$\ell_m[\eta\chi_L]-\ell_m[\eta] =\int\bigl((\chi_L^2-1)(\dot\eta^2+V_m\eta^2)+2\eta\chi_L\dot\eta\dot\chi_L+\eta^2(\dot\chi_L)^2\bigr)\,dt.$$ 
The first summand tends to zero by integrability of the energy density. On $  T_L  $ one has $ \eta^2=o(L)  $, hence $$\int_{T_L}\eta^2(\dot\chi_L)^2\le\frac1{L^2}\int_{T_L}\eta^2=o(1),$$
while Cauchy–Schwarz gives 
$$\int_{T_L}|2\eta\chi_L\dot\eta\dot\chi_L| \le\frac2L\|\eta\|_{L^2(T_L)}\|\dot\eta\|_{L^2(T_L)}=o(1)\cdot\|\dot\eta\|_{L^2(T_L)}\to0.$$ 
Thus $  \eta\chi_L\to\eta  $ in energy. Applying \eqref{e:gscalar} to the approximants and passing to the limit, one obtains $  \ell_m[\eta]\ge0  $ on unbounded components as well. Summing the identities over all connected components, we conclude that if $\eta$ vanishes at $  t_1,\dots,t_k  $, then $  \ell_m[\eta]\ge0  $. If moreover $  \ell_m[\eta]=0  $, both integrands in \eqref{e:gscalar} vanish on every component, so $  \eta/h  $ is constant and that constant must be zero. Hence $  \eta\equiv0  $. This implies that  $  \ell_m[\eta]>0  $ whenever $  \eta\not\equiv0  $ vanishes at the $  t_j  $.

We now prove the bound on the Morse index. Let $S$ be a finite-dimensional linear subspace of the energy space of $  Q_{m,g}^{(n)}  $ on which the form is nonpositive: $$Q_{m,g}^{(n)}\bigl[(c,d)\bigr]\le0\qquad\text{for every }(c,d)\in S.$$ Write $  \pi  $ for the continuous linear projection $  (c,d)\mapsto d  $. 
We first show that the restriction of $\pi$ to $S$ is injective.
 If $  (c,d)\in S  $ has $  d\equiv0  $, the expression for $  Q_{m,g}^{(n)}  $ reduces to 
$$Q_{m,g}^{(n)}[c,0]=\int_{\R}\bigl(\dot c^2+A_mc^2\bigr)\,dt.$$ 
The coefficient $  A_m  $ is strictly positive (this is where $  P<(n+2)^2  $ is used), so the integrand is a sum of squares and the integral vanishes only if $  c\equiv0  $. Thus $  \ker(\pi|_S)=\{0\}  $ and $$\dim\pi(S)=\dim S.$$ The image $  \pi(S)  $ is a finite-dimensional space of scalar functions belonging to the energy class of $  \ell_m  $. On that class the $  k  $ evaluation maps $$\delta_j\colon\eta\mapsto\eta(t_j),\qquad j=1,\dots,k,$$ 
are continuous linear functionals (every energy-class function is continuous, so the point values are well-defined). 

Since
\[
\dim S=\dim\pi(S),
\]
if $\dim S>k$, then the common kernel of the $k$ evaluation functionals on $\pi(S)$ is nontrivial. Hence there exists a nontrivial function
$d_*\in\pi(S)$ such that
\[
d_*(t_j)=0,\qquad j=1,\dots,k.
\]
Choose any pre-image $(c_*,d_*)\in S$. By \eqref{slow},
\[
Q_{m,g}^{(n)}[c_*,d_*]\ge \ell_m[d_*].
\]
Since $d_*$ belongs to the energy class of $\ell_m$ and vanishes at
$t_1,\dots,t_k$, the first part of the proof yields
\[
\ell_m[d_*]>0.
\]
Hence
\[
Q_{m,g}^{(n)}[c_*,d_*]>0,
\]
contradicting the assumption that
\[
Q_{m,g}^{(n)}\le0
\]
on $S$. Therefore
\[
\nu\bigl(Q_{m,g}^{(n)}\bigr)\le k.
\]
Finally, the ordering
\[
Q_m^{(n)}\ge Q_{m,g}^{(n)}
\]
implies
\[
\nu\bigl(Q_m^{(n)}\bigr)\le
\nu\bigl(Q_{m,g}^{(n)}\bigr)\le k,
\]
which completes the proof.
\end{proof}

A useful comparison relates the Fourier blocks indexed by $m$ and
$2n-m$.  If $2\le m<n$, define
\[
m^*=2n-m.
\]
Since
\[
(n-m)^2=(n-m^*)^2,
\]
the only difference between the corresponding scalar potentials comes
from the first term in $A_m$. Indeed,
\[
A_{m^*}
=
(n+m^*)^2-P
>
(n+m)^2-P
=
A_m.
\]
Moreover, the function
\[
A\longmapsto\frac{AW}{A+W}
\]
is increasing for $A>0$. Hence
\begin{equation}\label{refl}
 V_{m^*}>V_m.
\end{equation}
Consequently, every positivity result established for the scalar problem associated with the mode $m$ automatically extends to the reflected mode $2n-m$.

Several of the positivity conditions arising below reduce to proving that certain explicit polynomials are nonnegative. For this purpose we shall
repeatedly use the following elementary criterion based on the
Bernstein basis.

\begin{lemma} 
\label{le3.3}
If a function $p$ satisfies
\[
 p(s)=\sum_{j=0}^d b_j\binom dj s^j(1-s)^{d-j},
 \qquad 0\leq s\leq1,
\]
where all the Bernstein coefficients $b_j$ satisfy $b_j\ge0$,
$p\ge0$ on $[0,1]$.

\end{lemma}

\begin{proof}
The Bernstein basis functions are nonnegative and form a partition of unity on $[0,1]$. Hence
\[
p(s)\ge \min_{0\le k\le d} b_k,
\qquad 0\le s\le1,
\]
which proves the lemma.
 Therefore, if all $b_k$ are non-negative, then in the interval $[0,1]$, $p$ will be bounded from below by the minimum of all $b_k$.
\end{proof}

 It is worth mentioning how the Bernstein coefficients are computed, because this formula will be used below.  
Let $p(s)=\sum_{j=0}^dp_js^j$.  Its degree-$d$ Bernstein coefficients
are
\begin{equation}\label{e:bern}
 b_k=\sum_{j=0}^kp_j\frac{\binom{k}{j}}{\binom{d}{j}},
 \qquad 0\leq k\leq d.
\end{equation}
Indeed,
$s^j=\binom dj^{-1}\sum_{k=j}^d\binom kj
\binom dk s^k(1-s)^{d-k}$. 
To treat a subinterval,
substitute $s=\alpha+(\beta-\alpha)u$ and apply
\eqref{e:bern}.
For a polynomial defined on the half-line,
$P(z)=\sum_{j=0}^dp_jz^j$, put
$z=w/(1-w)$.  The compactified polynomial is
$$
 \widehat P(w)=(1-w)^dP\left(\frac{w}{1-w}\right)
 =\sum_{j=0}^dp_jw^j(1-w)^{d-j}.
$$
Thus $P>0$ on $[0,\infty)$ whenever the Bernstein coefficients of
$\widehat P$ on the listed subdivision of $[0,1]$ are positive.

\bigskip
We are now ready to analyze the nondegeneracy and the Morse index of the degree-two and degree-three vortices.

The key step is the construction of explicit supersolutions to which Proposition~\ref{p:trace} can be applied.
We briefly describe the guiding principle behind this construction.

The supersolutions are sought within simple finite-dimensional ansatzes that 
capture the correct asymptotic behavior at the two endpoints and vanish at a prescribed trace point
$x_0$. Near $x=0$, the positive solution of
$-\partial_t^2+V_m$ behaves like
$x^{|n-m|/2}$ (or is asymptotically constant when $m=n$). We multiply
this leading behavior by a linear factor vanishing at $x_0$, and use an
analogous construction near infinity.
The exponents and the rational choice of $x_0$ are then selected so that, after clearing the positive
denominators, the quantity
\[
(-\partial_t^2+V_m)h_\pm
\]
reduces to an explicit polynomial whose positivity can be verified
directly. 
For example, this procedure leads to the simple choices
\[
1-\frac{x}{4},
\qquad
1-\frac{4}{x},
\]
for the block $(n,m)=(2,2)$, and to similarly elementary supersolutions
for the block $(n,m)=(3,2)$.
Once the finite-dimensional ansatz has been fixed, the construction becomes completely algorithmic:
 the positivity condition is reduced to verifying the positivity of an explicit polynomial, or 
 equivalently, to a finite system
 of linear inequalities for its Bernstein coefficients. 

The block $(n,m)=(3,3)$ requires a more flexible ansatz,
since low-degree trial functions do not yield a positivity polynomial
that can be verified directly. We therefore consider
\[
h_+(y)
=
(y-1)y^{-10}R(y-1),
\qquad
R(z)
=
\sum_{j=0}^8 a_j z^j.
\]
Because the differential operator is linear, the coefficients of
\[
y^{10}D E_3(-\partial_t^2+V_3)h_+
\]
depend linearly on the coefficients
$a_0,\ldots,a_8$. We compactify the half-line by introducing
\[
w=\frac{z}{1+z},
\]
express the resulting polynomial on each subinterval of
$[0,1]$ in the Bernstein basis, and require all Bernstein coefficients
to be nonnegative. 
This leads to a finite system of linear
inequalities for the coefficients
$a_0,\ldots,a_8$.
After imposing the normalization
$a_0=1$, we solve the corresponding linear programming problem by
maximizing the smallest Bernstein coefficient.
The resulting numerical optimizer is then approximated by nearby rational coefficients, which are finally rescaled to integers.  All symbolic and numerical computations
were carried out in \emph{Mathematica}. 
Neither the coefficients $a_j$ nor the splitting point $x_0$ are unique.

\bigskip

\subsection{Degree-2 solution}
For the standard degree-two vortex solution, use the lower model $g=\phi_{C_*}$ and put
$D(x)=x^2+4x+C_*$.  In \eqref{av},
\begin{equation*}
 V=-P+\frac{AW}{A+W}=\frac{xN(x)}{D(x)E(x)},
\end{equation*}
where
\begin{align*}
 E(x)&=x^3+12x^2+(64-C_*)x+16C_*,\\
 N(x)&=8x^4+(16-2C_*)x^3+(8C_*-256)x^2
 +(C_*^2-128C_*)x-16C_*^2.
\end{align*}
We split the analysis at the point $x=4$.
Consider the functions
 \begin{equation*}
 h_-=1-\frac x4\quad(0<x<4),\qquad
 h_+=1-\frac4x\quad(x>4)
\end{equation*}
which are positive on the indicated intervals. 

Using
\[
\partial_t^2x=4x,
\qquad
\partial_t^2x^{-1}=4x^{-1},
\]
and the identity
\[
C_*^2=60C_*-144,
\]
we obtain
\begin{align}
 \frac{4DE}{x^2}(-\partial_t^2+V)h_-
={}&(960C_*-2304)+(196C_*+144)x\notag\\
&+(768-16C_*)x^2+(80+2C_*)x^3-4x^4,\notag\\
 x DE(-\partial_t^2+V)h_+
={}&c_0+c_1y+c_2y^2+c_3y^3+c_4y^4+c_5y^5+8y^6,
\label{p2r}
\end{align}
where $y=x-4$ and
\begin{align*}
 (c_0,c_1,c_2)&=(234496+25856C_*,\ 209920-768C_*,\ 63616-1248C_*),\\
 (c_3,c_4,c_5)&=(11888-164C_*,\ 1856-24C_*,\ 192-2C_*).
\end{align*}
The first polynomial is positive on $0<x<4$ after grouping its terms
 and using
$57<C_*<58$.  Indeed,
\begin{align*}
 &(196C_*+144)x+(768-16C_*)x^2
 \geq x(132C_*+3216)>0,\\
 &(80+2C_*)x^3-4x^4
 \geq(64+2C_*)x^3>0,
\end{align*}
while the constant term is obviously positive. 
For the second polynomial,
all terms except
$c_2y^2$ have positive coefficients.  
Moreover,
\[
 c_2^2-4c_1c_3=-6086852608+78663680C_*<0.
\]
Thus $c_1+c_2y+c_3y^2>0$, and \eqref{p2r} is positive.
Proposition~\ref{p:trace} yields $\nu(Q_2^{(2)})\leq1$.

To obtain a negative direction, take $a=(r^2+2)^{-1}$ and $b=0$.  From the upper
barrier $f_2<\phi_{24}$, $y=r^2$, and $z=y+2$,
\begin{equation}\label{e:neg2}
 Q_2^{(2)}[a,0]<\frac1{12}+\frac12J,
\quad
J=\int_2^\infty\frac{z^2-8z-12}{z^2(z^2+20)}\dd z.
\end{equation}
Explicitly, we have
\[
 J=\frac{2\pi}{5\sqrt5}-\frac15\log6-\frac3{10}
 -\frac4{5\sqrt5}\arctan\frac1{\sqrt5}.
\] 
Hence $Q_2^{(2)}[a,0]<0$. 
Therefore
\begin{equation}\label{end2}
 \ind Q_2^{(2)}=1,\qquad \ker_{L^2}Q_2^{(2)}=\{0\}.
\end{equation} 
This finishes the analysis of the degree 2 solution.
 
\subsection{Degree-3 solution}
For degree 3 vortex solution, by Proposition~\ref{p:bd3} and \eqref{bd3}, the lower bound of $f_3$
may be taken with $g=g_{4000}$.  It has
\[
 D=x^3+9x^2+99x+4000,\quad p_0=x(9x^2+99x+4000),
\]
and $E_m,N_m$ are given by
\eqref{spoly}--\eqref{npoly}.  
\bigskip

\textit{The mode $m=2$ case}

For Fourier mode $m=2$, let us split the interval $(0,+\infty)$ at
$x=12$ and take
\begin{equation}\label{h32}
 h_- =\sqrt{x}\left(1-\sqrt{x/12}\right),\qquad
 h_+=1-(12/x)^2.
\end{equation}
Both functions in \eqref{h32} are positive on their respective
intervals.  Writing $s=\sqrt{x/12}$ on the left interval $(0,12)$, by direct calculation, we have
\begin{align*}
(-\partial_t^2+V_2)h_-=~&\frac{-(\sqrt{12}s-8\sqrt{3}s^2)DE_2+N_2(\sqrt{12}s-\sqrt{12}s^2)}{DE_2}\\
=~&\frac{\sqrt{12}s((N_2-DE_2)-s(N_2-4DE_2))}{DE_2}.
\end{align*}
Since $\frac{s}{DE_2}>0$ for all $s\in(0,1)$, the sign of
\[
(-\partial_t^2+V_2)h_-
\]
is determined by
\[
 N_2-DE_2-s(N_2-4DE_2).
\]
Substituting the explicit expressions for
$D$, $E_2$, and $N_2$, a straightforward computation yields
\begin{equation*}
p(s):= (N_2-DE_2)-s(N_2-4DE_2)=\sum_{j=0}^{17}a_js^j,
\end{equation*}
where
\begin{align*}
(a_0,\ldots,a_{17})=~&(
0,1200000000,-4800000000,4936800000,-547200000\\
&~1088506800,-2165227200,3008918016,772436736\\
&~-449390592,-2051122176,2326579200,-349360128\\
&~573308928,250822656,-143327232,0,0).
\end{align*}
In fact, $p(s)$ is a polynomial of degree $15$. To check the positivity of $p(s)$ for $s\in(0,1]$ we regard it as a polynomial of degree elevation to $17$ in the Bernstein basis. Now we make a subdivision for $p(s)$ at $s=\frac12$. To be precise, we define
\begin{align*}
F_1(s)=p\left(\frac{s}{2}\right),\quad F_2(s)=p\left(\frac{s+1}{2}\right)\quad \mbox{for}\quad s\in(0,1).
\end{align*}
Then we write both $F_1(s)$ and $F_2(s)$ into the form presented in Lemma \ref{le3.3}, and find that the corresponding Bernstein coefficients respectively are
\begin{align*}
(b_0,\ldots,b_{17})=~&\left(
0,\frac{600000000}{17},\frac{1050000000}{17},\frac{1365427500}{17},\frac{10930260000}{119},\right.\\
&~~\frac{1203491031675}{12376},\frac{1205131158375}{12376},\frac{905516152011}{9724},\frac{1037421015588}{12155},\\
&~~\frac{3642951228381}{48620},\frac{1219896550329}{19448},\frac{43872642897}{884},\frac{56537153808}{1547},\\
&\left.~~\frac{3322246497}{136},\frac{9721678323}{680},\frac{989258631}{136},\frac{78733053}{17},\frac{15719103}{2}\right),
\end{align*}
and
\begin{align*}
(b_0,\ldots,b_{17})=~&\left(
\frac{15719103}{2},\frac{188491698}{17},\frac{2745396951}{136},\frac{24920175327}{680},\right.\\
&~~\frac{42282454101}{680},\frac{611294494635}{6188},\frac{921136958121}{6188},\frac{4187283074067}{19448},\\
&~~\frac{14671258945557}{48620},\frac{10035264775977}{24310},\frac{414722539035}{748},\frac{9091632825789}{12376},\\
&~~\frac{1704564617349}{1768},\frac{2993713542579}{2380},\frac{55650714147}{34},\frac{36220677429}{17},\\
&\left.~~\frac{47292851064}{17},3651745008\right).
\end{align*}
It is not difficult to see that for both $F_1(s)$ and $F_2(s)$, the Bernstein coefficients are positive except for the first coefficient of $F_1(s)$. Therefore, $F_1(s)>0$ for $s\in(0,1]$ and $F_2(s)>0$ for $s\in[0,1]$. Consequently, we have $p(s)>0$ for $s\in(0,1]$. Thus, we have checked the positivity of $(-\partial_t^2+V_2)h_-$ for $x\in(0,12]$. In the right interval $(12, +\infty)$,
\begin{equation*}
 x^2DE_2(-\partial_t^2+V_2)h_+
 =(x^2-144)N_2+2304DE_2.
\end{equation*}
In the new variable $y=x-12$, this is
\begin{align*}
&2804540166144+360893814144y+29582245840y^2+1342870192y^3\\
&\qquad+12378097y^4+722684y^5+338916y^6+25918y^7+772y^8+8y^9,
\end{align*}
which is positive coefficientwise.
\bigskip

\textit{The mode $m=3$ case}

For Fourier mode $m=3$, we shall split the interval $(0, +\infty)$ into two parts at the trace point $x_0=17/4$.  On the left part $(0, x_0)$, let us take
\begin{equation}\label{h33}
 h_-=(17/4-x)(1-x/68).
\end{equation}
This function is positive on this left interval.  After
degree elevation to $10$, the  Bernstein coefficients of the function 
$$J_1(x):=DE_3(-\partial_t^2+V_3)h_-$$ on $[0,17/4]$ are all non-negative, 
so the inequality is strict in the interior. See Figure \ref{fig1} for the graph of $J_1$.
\begin{figure}[htbp]
  \centering 
  \includegraphics[width=1\textwidth]{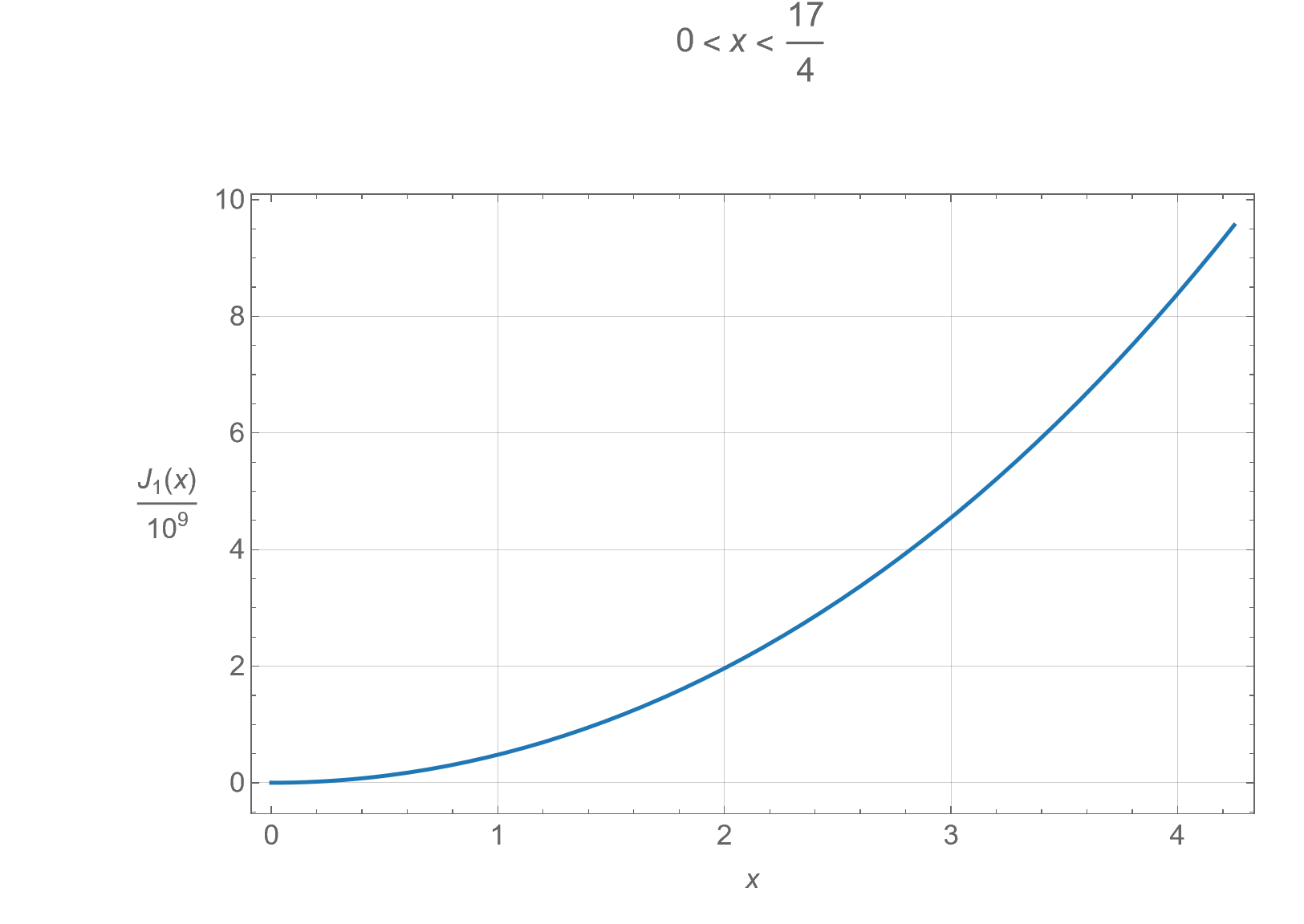}
  \caption{The graph of $J_1(x)/10^9$ on $0<x<17/4$.}
  \label{fig1}
\end{figure}

On the right part $(x_0, +\infty)$, we put
$y=4x/17$, $z=y-1$ and
\begin{align*}
 R(z)={}&36+336z+1357z^2+3815z^3+4442z^4\notag\\
 &+7378z^5+2275z^6+2478z^7+z^8.
\end{align*}
Then we choose a supersolution in this interval to be \begin{equation}\label{m3h}h_+=zy^{-10}R(z).\end{equation}
We choose the degree of $R$ to be $8$, but perhaps suitable degree $7$ polynomial can also be found. We will not pursue this issue here. 

We have
\[
 8192y^{10}DE_3(-\partial_t^2+V_3)h_+
 =\sum_{j=0}^{16}c_jz^j:=J_2(z),
\]
where
\begin{align*}
(c_0,\ldots,c_{16})={}&(
272315510507160,
7548510753099936,
-147414247717765662,\\
&894417885521795655,
-2455279302164126424,
3520196119680798726,\\
&-2730343202953398050,
1144374981300423899,
-238274168212913543,\\
&6800854942671690,
11926229467654612,
-3618122902015173,\\
&473509221032176,
-35467122492360,
-2700164774612,\\
&1046484303995,
2872370711).
\end{align*}
Compactification of the $z$ interval using $w=z/(1+z)$ gives us a new polynomial $$\widehat J_2(w):=(1-w)^{16} J_2\left(\frac{w}{1-w}\right), w\in [0,1],$$ and subdivide the interval $[0,1]$ at
\[
0,\frac14,\frac12,\frac34,\frac{13}{16},\frac78,1.
\]
gives the minimum Bernstein coefficients which are all positive. In fact, we can follow the arguments as we did for the Fourier mode $m=2$ case, we set
\begin{align*}
q_1(w)=\hat J_2\left(\frac{w}{4}\right), q_2(w)=\hat J_2\left(\frac{w+1}{4}\right) ,
q_3(w)=\hat J_2\left(\frac{w+2}{4}\right),\\
q_4(w)=\hat J_2\left(\frac{w+12}{16}\right),
q_5(w)=\hat J_2\left(\frac{w+13}{16}\right),
q_6(w)=\hat J_2\left(\frac{w+7}{8}\right),
\end{align*}
for $w\in(0,1)$. By direct calculation, we can see that all the Bernstein coefficients of $q_1(w),\cdots,q_6(w)$ are strictly positive. Therefore, we get $\hat J_2(w)>0$ for $w\in[0,1]$ and it implies the positivity of $J_2(z)$ for $x\geq \frac{17}{4}$.  See Figure \ref{fig2} for the graph of $\widehat{J_2}$.
\begin{figure}[htbp]
  \centering
  \includegraphics[width=1\textwidth]{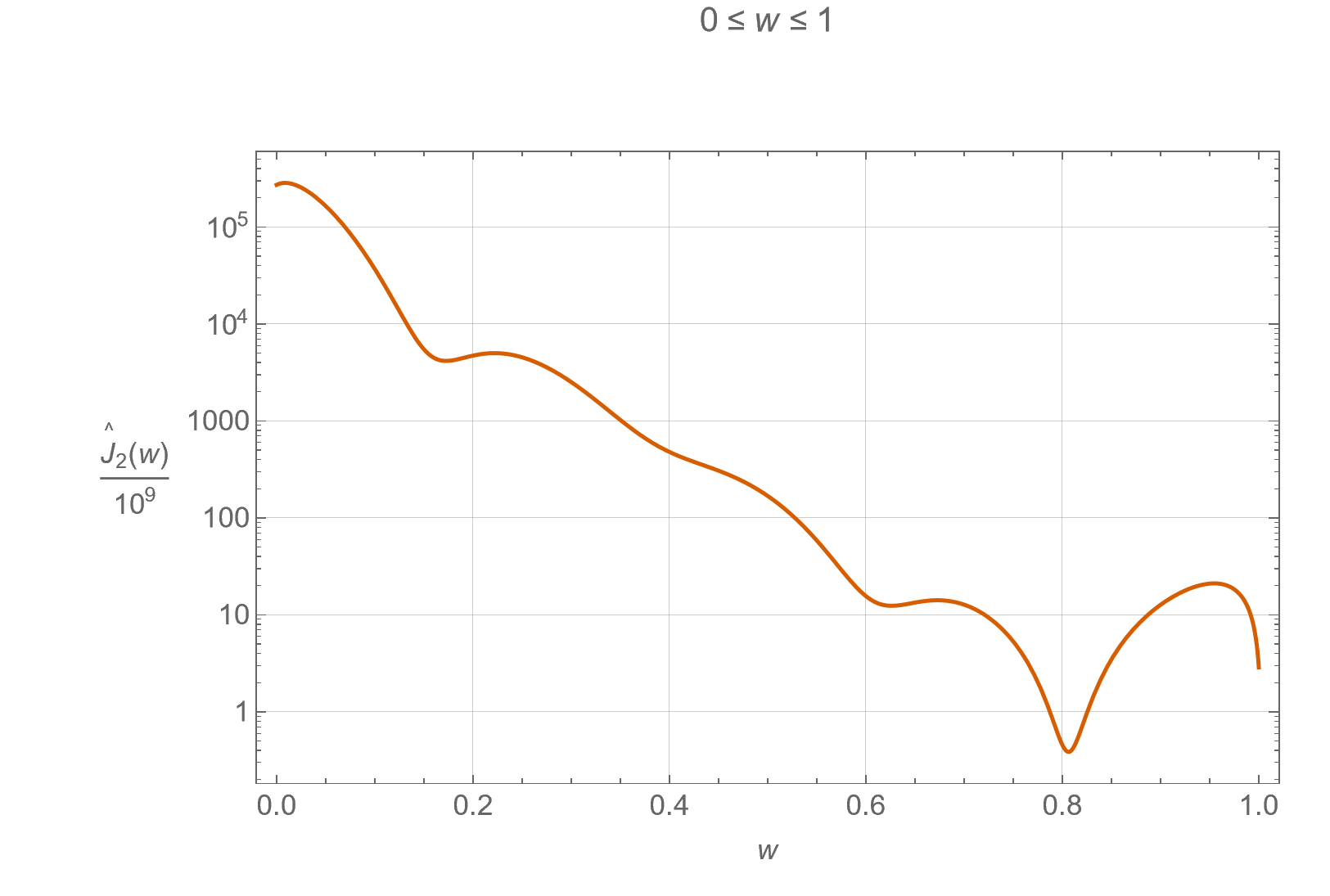}
  \caption{The graph of $\widehat{J}_2(w)/10^9$ on $0\leq w\leq1$.
  The vertical axis is logarithmic.}
  \label{fig2}
\end{figure}
Proposition~\ref{p:trace} then gives
$\nu(Q_m^{(3)})\leq1$ for $m=2,3$. We then obtain from the comparison principle
\eqref{refl}  the same bound for $m=4$.

It remains to find negative directions of the quadratic form by constructing suitable test functions. This will provide a lower bound of the index. The upper barriers $g_{595}$ and $g_{243}$
give, respectively,
\[
 Q_m^{(3)}\left[\frac{r}{(r^2+7)^2},0\right]<0 
 \quad \text{for} \quad m=2,4, \]
and \[
   Q_3^{(3)}[(r^2+3)^{-3/2},0]<0.
\]
To see this, we use
$g_{243}^2<(x/(x+3))^3$. Then we compute, with
$a=(x+3)^{-3/2}$ and $r\dd r=\dd x/2$,
\begin{align*}
 Q_3^{(3)}[a,0]
 &<\int_0^\infty\left\{
 \frac{9x}{2(x+3)^5}+\frac{x^3}{(x+3)^6}
 -\frac1{2(x+3)^3}\right\}\dd x< 0.
\end{align*}
We then conclude that
\begin{equation}\label{end3}
 \ind Q_m^{(3)}=1,\qquad \ker_{L^2}Q_m^{(3)}=\{0\},\qquad m=2,3,4.
\end{equation}
  
The analysis of the degree-three vortex is now complete.  

\medskip  
For higher-degree vortex solutions, the corresponding construction is expected to require higher-degree polynomials and will therefore become substantially more involved.

\bigskip
Now let us summarize all these results and finally prove the main results in this paper.
\begin{proof}[Proof of Theorems~\ref{t:main} and \ref{t:block}]
The real Fourier decomposition is orthogonal and has two isomorphic real
angular copies for every $m\geq1$.  

The preceding results show that 
the bounded kernel in the $m=0$ block is $\R iu_n$, the two $m=1$ copies are the
translations, and every $m\geq2n-1$ block is strictly positive.
Any bounded zero modes in the intervening blocks are radial $L^2$
modes.  They are excluded by \eqref{end2} and
\eqref{end3}.  This proves \eqref{bker} and the blockwise
assertions in Theorem~\ref{t:block}.
In particular, \eqref{end2}--\eqref{end3} establish the claims
announced in \eqref{tab2}--\eqref{tab3}.

None of the three symmetry modes lies in $L^2(\R^2)$: $|iu_n|\to1$,
while $|\nabla u_n|=n/r+O(r^{-3})$.
  The block classification therefore
also proves \eqref{l2ker}.

Recall that only the intermediate  blocks with $1<m<2n-1$ may have negative spectrum.  Summing their radial
indices and multiplying by the two real angular copies yields the following table:
\[
\begin{array}{c|c|c}
n&\text{sum of radial indices}&\text{full real index}\\ \hline
2&1&2\\
3&3&6
\end{array}
\]
This finishes the Morse index computation.   
\end{proof}

\bigskip

\begin{center}
{\bf Acknowledgement}
\end{center}

The research of M.del Pino is supported by the Royal Society Research Professorship grant RP-R1-180114 and by the ERC/UKRI Horizon Europe grant ASYMEVOL, EP/Z000394/1. The research of Y. Liu is supported by  NSFC No. 12471204.  The research of J.Wei is partially supported by GRF of RGC of Hong Kong entitle ``On critical and supercritical Fujita equation".  The research of W. Yang is partially supported by National Key R\&D Program of China 2022YFA1006800, NSFC No. 12531010, FDCT No. 0070/2024/RIA1,  Multi-Year Research Grant No. MYRG-GRG2024-00082-FST-UMDF, MYRG-GRG2025-00051-FST and UMDF No. TISF/2025/006/FST. 
The authors acknowledge the use of ChatGPT (OpenAI) as an AI tool. All mathematical arguments and proofs in the final manuscript were checked and written by the authors.

\end{document}